\documentclass{article}
\usepackage{graphicx} 

\usepackage{cite}
\usepackage{amsmath,amssymb,amsfonts}
\usepackage{algorithmic}
\usepackage{graphicx}
\usepackage{textcomp}
\def\BibTeX{{\rm B\kern-.05em{\sc i\kern-.025em b}\kern-.08em
    T\kern-.1667em\lower.7ex\hbox{E}\kern-.125emX}}

\usepackage{graphicx} 
\usepackage{amsmath} 
\usepackage{amssymb}  
\usepackage{bm}
\usepackage{dsfont}
\usepackage{mathtools}
\mathtoolsset{showonlyrefs}
\usepackage{upgreek}
\usepackage{amsthm}
\usepackage{accents}
\usepackage{verbatim}
\usepackage{multirow}
\usepackage{paralist}
\usepackage{todonotes}
\usepackage{xcolor}
\usepackage{booktabs}
\usepackage{caption}
\usepackage{subcaption}
\usepackage[ruled, english]{algorithm2e}

\usepackage{pifont}

\newcommand{\mf}{\mathbf}

\newcommand{\cov}{\normalfont\textsf{\footnotesize cov}}

\newcommand{\sg}{\normalfont\textsf{\footnotesize sgn}}

\newcommand{\vect}{\normalfont\textsf{\footnotesize vec}}
\newcommand{\diag}{\normalfont\textsf{\footnotesize   diag}}
\newcommand{\col}{\normalfont\textsf{\footnotesize   col}}

\newtheorem{corollary}{Corollary}
\newtheorem{lemma}{Lemma}
\newtheorem{assumption}{Assumption}
\newtheorem{remark}{Remark}
\newtheorem{theorem}{Theorem}

\title{A Note on Optimal Distributed State Estimation \\ for Linear Time-Varying Systems\footnote{This work was supported via projects PID2021-124137OB-I00 and TED2021-130224B-I00 funded by MCIN/AEI/10.13039/501100011033, by ERDF A way of making Europe and by the European Union NextGenerationEU/PRTR, by the Gobierno de Aragón under Project DGA T45-23R, by the Universidad de Zaragoza and Banco Santander, by the Consejo Nacional de Ciencia y Tecnología (CONACYT-Mexico) with grant number 739841, and by Spanish grant FPU20/03134.}}
\author{Irene Perez Salesa, Rodrigo Aldana-López, Carlos Sagüés\footnote{The authors are with Departamento de Informática e Ingeniería de Sistemas (DIIS) and Instituto de Investigación en Ingeniería de Aragón (I3A), Universidad de Zaragoza, María de Luna 1, 50018, Zaragoza, Spain (e-mails: i.perez@unizar.es, rodrigo.aldana.lopez@gmail.com, csagues@unizar.es).}}
\date{}

\begin{document}

\maketitle

\begin{abstract}
\noindent In this technical note, we prove that the ODEFTC algorithm constitutes the first optimal distributed state estimator for continuous-time linear time-varying systems subject to stochastic disturbances. Particularly, we formally show that it is able to asymptotically recover the performance, in terms of error covariance of the estimates at each node, of the centralized Kalman-Bucy filter, which is known to be the optimal filter for the considered class of systems. Moreover, we provide a simple sufficient value for the consensus gain to guarantee the stability of the distributed estimator.

\end{abstract}

\section{Introduction}

This work deals with distributed state estimation in sensor networks. This topic has attracted considerable attention in recent years, having applications in fields such as monitoring and control of large-scale systems, target tracking or navigation \cite{He2020}.
Particularly, we consider continuous-time linear systems subject to stochastic noise. 
The goal for the nodes of the sensor network is to collaboratively obtain estimates of the state, by using their local noisy measurement and communication with neighbors. In addition, it is desirable to achieve similar performance to the optimal centralized filter for this class of systems, which is the Kalman-Bucy filter \cite{Kalman1961}.

Several works have tackled this problem through consensus-based solutions, mainly focusing on the linear time-invariant (LTI) case \cite{Ren2005, Kim2016, Wu2016, Ji2017, Battilotti2020, ifac2023, ecc24}. For linear time-varying (LTV) systems, \cite{Olfati-Saber2007} proposes a solution under a restrictive observability assumption, while \cite{Ren2018} provides a solution under collective observability. However, this proposal only recovers the solution of the centralized filter in absence of noise, which contradicts the usual assumptions in the context of Kalman filtering. To solve this problem, the ODEFTC algorithm, which stands for ``Optimal Distributed Estimation based on Fixed-Time Consensus", was introduced in \cite{odeftc}. This algorithm generalizes the proposal from \cite{Battilotti2020} to the LTV case. Despite simulation experiments validating that ODEFTC recovers a similar performance to the optimal centralized filter at each node, this aspect was only formally proven for the LTI case, based on results from \cite{Battilotti2020}. Therefore, analyzing the optimality of ODEFTC in the LTV case remains an open problem, which we address in this technical note. Furthermore, we provide a sufficient choice of the consensus gain to guarantee the stability of the estimator in the LTV context.

\section{Preliminaries}

\subsection{Notation}
The $n \times n$ identity matrix is represented with $\mf{I}_n$, the $n \times 1$ vector of ones is $\mathds{1}_n$ and the $n \times n$ matrix of ones is $\mf{U}_n$. $\mf{A} \succ (\succeq)\ \mf{0}$ means that $\mf{A}$ is positive definite (semidefinite, respectively). The notation $\| \bullet \|$ represents the Euclidean norm when applied to a vector, and the induced norm when applied to a matrix. Expectation and covariance are denoted by $\mathds{E}\{\bullet \}$ and $\cov\{\bullet\}$. The notation $\diag_{i=1}^N(\bullet_i)$ denotes the block-diagonal composition of matrices indexed by $i=1,\dots,N$, and $\col_{i=1}^N(\bullet_i)$ is their column composition, while $\vect(\bullet)$ is the vectorization operator, which transforms an $n \times m$ matrix into a $nm \times 1$ vector. The maximum and minimum eigenvalues of a matrix are denoted with $\lambda_{\max}(\bullet)$ and $\lambda_{\min}(\bullet)$.

\subsection{The ODEFTC Algorithm}
The ODEFTC algorithm provides a distributed filter for sensor networks observing a dynamical system affected by stochastic noise, described by
\begin{equation}\label{eq:sys}
    \dot{\mf{x}}(t) = \mf{A}(t) \mf{x}(t) + \mf{w}(t), \ t \geq 0,
\end{equation}
where $\mf{x}(t) \in \mathbb{R}^n$ is the state vector and $\mf{w}(t) \in \mathbb{R}^{n}$ is zero-mean Gaussian noise with covariance $\mf{W}(t) \succeq \mf{0}$. The initial state is normally distributed, with mean $\mf{x}_0$ and covariance $\mf{P}_0$.

The sensor network is formed by $N$ nodes, with its communication described by an undirected, connected graph $\mathcal{G} = (\mathcal{V}, \mathcal{E})$. The set $\mathcal{V} =\{1, \dots, N\}$ contains the nodes, while $\mathcal{E}\subseteq \mathcal{V} \times \mathcal{V}$ is the edge set representing the communication links. The graph topology is described by its Laplacian matrix $\mf{Q}_\mathcal{G}$, with $\lambda_\mathcal{G}$ being its algebraic connectivity. The set of neighbors of an arbitrary node $i$ is $\mathcal{N}_i = \{j \in \mathcal{V} \ : \ (i,j) \in \mathcal{E}\}$.
Each node $i$ takes a measurement $\mf{y}_i(t) \in \mathbb{R}^{n_{\mf{y}_i}}$ given by
\begin{equation}\label{eq:meas}
    \mf{y}_i(t) = \mf{C}_i(t) \mf{x}(t) + \mf{v}_i(t), \ \forall t \geq 0,
\end{equation}
with $\mf{v}_i(t) \in \mathbb{R}^{n_{\mf{y}_i}}$ being white Gaussian noise with covariance $\mf{R}_i(t) \succ \mf{0}$.

The goal for each sensor node is to obtain an estimate $\hat{\mf{x}}_i(t) \in \mathbb{R}^n$ of the state $\mf{x}(t)$ and an estimated error covariance matrix $\mf{P}_i(t)$, using its measurement $\mf{y}_i(t)$ and communication with neighboring nodes. Then, the ODEFTC algorithm run by each node $i \in \mathcal{V}$ is given by
\begin{subequations}\label{eq:dkb}
\begin{align}
    \mf{K}_i(t) &= N \mf{P}_i(t) \mf{C}_i(t)^\top \mf{R}_i(t)^{-1} \label{eq:dkb-k}\\
    \dot{\hat{\mf{x}}}_i(t) &= \mf{A}(t)\hat{\mf{x}}_i(t) + \mf{K}_i(t) \left(\mf{y}_i(t) - \mf{C}_i(t)\hat{\mf{x}}_i(t)\right)\nonumber \\&\quad + \kappa \mf{P}_i(t) \textstyle \sum_{j\in \mathcal{N}_i} (\hat{\mf{x}}_j(t) - \hat{\mf{x}}_i(t))\label{eq:dkb-x}\\ 
    \dot{\mf{P}}_i(t) &= \mf{A}(t)\mf{P}_i(t) + \mf{P}_i(t) \mf{A}(t)^\top + \mf{W}(t)      - \mf{P}_i(t)\hat{\mf{Z}}_i(t)\mf{P}_i(t) \label{eq:dkb-p}\\
    \hat{\mf{Z}}_i(t) &= N\mf{C}_i(t)^\top \mf{R}_i(t)^{-1}\mf{C}_i(t) - \mf{Q}_i(t) \label{eq:dkb-z1}\\
    \dot{\mf{Q}}_i(t) &= \alpha \textstyle \sum_{j\in \mathcal{N}_i} \phi(\hat{\mf{Z}}_i(t) - \hat{\mf{Z}}_j(t), \xi, \gamma),\label{eq:dkb-z2}
\end{align}
\end{subequations}
where the function
$\phi(\bullet, \xi, \gamma) = (|\bullet |^{1-\gamma} + |\bullet |^{1+\gamma} + \xi)\sg(\bullet)$
is applied element-wise to the matrices, noting that $\sg$ is the ``sign" operator, $\kappa, \alpha, \gamma, \xi > 0$ are design parameters, and \eqref{eq:dkb-z1}-\eqref{eq:dkb-z2} constitutes the auxiliary fixed-time consensus algorithm on which the estimator is based. The initialization at $t=0$ fulfills $\sum_{i=1}^N \mf{Q}_i(0) = \mf{0}$, which can be easily complied if $\mf{Q}_i(0) = \mf{0}, \, \forall i \in \mathcal{V}$, and, ideally, $\hat{\mf{x}}_i(0) = \mf{x}_0$, $\mf{P}_i(0) = \mf{P}_0$, although a different initialization can be used if $\mf{x}_0, \, \mf{P}_0$ are unknown in practice. The nodes communicate the variables $\hat{\mf{x}}_i(t)$ and $\hat{\mf{Z}}_i(t)$ with their neighbors.

\subsection{Assumptions}

We recall the following assumptions for ODEFTC \cite{odeftc}. 
\begin{assumption}\label{assum:observable}
    The pair $(\mf{A}(t), \mf{C}(t))$ is uniformly completely observable, with $\mf{C}(t) = \col_{i=1}^N(\mf{C}_i(t))$ $,\forall t\geq 0$.
\end{assumption}
\begin{assumption}\label{assum:bounds}
    There exist constants $0\leq a,\, c,\, w_1 \leq w_2,\, r_1 \leq r_2$ such that $\|\mf{A}(t)\|\leq a,\, \|\mf{C}(t)\|\leq c,\, w_1\leq\|\mf{W}(t)\|\leq w_2,\, r_1\leq \|\mf{R}(t)\|\leq r_2$, $\forall t\geq 0$.    
\end{assumption}
\begin{assumption}\label{assum:bounds-Z}
    Letting $\mf{Z}_i(t) = N\mf{C}_i(t)^\top \mf{R}_i(t)^{-1}\mf{C}_i(t)$, there exists a known $L \geq 0$ such that $\| \dot{\mf{Z}}_i(t)\| \leq L, \ \forall t\geq 0$.
\end{assumption}
Note that Assumptions \ref{assum:observable} and \ref{assum:bounds} are commonly required for Kalman filtering, while the constraint in Assumption \ref{assum:bounds-Z} is given in aggregate form for convenience and can be reasonably fulfilled in many cases. For example, it holds trivially for LTI systems and for LTV systems with a static sensor model.

\subsection{Optimal Centralized Filter}
Let $\mf{y}(t) = \col_{i=1}^N(\mf{y}_i(t)) = \mf{C}(t) \mf{x}(t) + \mf{v}(t)$ be the aggregate measurement from all the nodes, with $\mf{C}(t)$ from Assumption \ref{assum:observable} and $\mf{v}(t) = \col_{i=1}^N(\mf{v}_i(t))$ being the aggregate noise vector with covariance $\mf{R}(t) = \diag_{i=1}^N(\mf{R}_i(t))$.
Then, the optimal centralized solution to obtain minimum mean-square-error estimates of system \eqref{eq:sys} under the measurement $\mf{y}(t)$ is given by the Kalman-Bucy filter \cite{Kalman1961},
\begin{subequations}\label{eq:ckb}
\begin{align}
    \mf{K}(t) &= \mf{P}(t) \mf{C}(t)^\top \mf{R}(t)^{-1}\label{eq:ckb-k} \\
    \dot{\hat{\mf{x}}}(t) &= \mf{A}(t)\hat{\mf{x}}(t) + \mf{K}(t) \left(\mf{y}(t) - \mf{C}(t)\hat{\mf{x}}(t)\right) \label{eq:ckb-x}\\
    \dot{\mf{P}}(t) &= \mf{A}(t) \mf{P}(t) + \mf{P}(t) \mf{A}(t)^\top + \mf{W}(t)  - \mf{P}(t)\mf{C}(t)^\top\mf{R}(t)^{-1}\mf{C}(t)\mf{P}(t) \label{eq:ckb-p}
\end{align}
\end{subequations}
where $\mf{K}(t)$ is the filtering gain, $\hat{\mf{x}}(t)$ is the state estimate and $\mf{P}(t) = \cov\{\mf{x}(t) - \hat{\mf{x}}(t)\}$ is its error covariance matrix.

\begin{remark}\label{rem:p-ckb}
    Assumptions \ref{assum:observable} and \ref{assum:bounds} imply that $\mf{P}(t)$ is uniformly upper bounded \cite[Lemma 3.2]{Anderson1971}. 
    Moreover, $\mf{P}(t) \succ \mf{0}$ generally holds, since having a singular matrix $\mf{P}(t)$ implies that some of the states can be known exactly, with no need to estimate them (see additional discussions in \cite{Anderson1971}). Then, $\mf{P}(t)^{-1}$ exists and also satisfies a Riccati equation, which means that its upper bound can be found by the same arguments, so that $p_1 \mf{I}_n \preceq \mf{P}(t) \preceq p_2 \mf{I}_n, \, \forall t \geq 0,$ for some constants $p_1,\, p_2 > 0$. Thus, functions of the form $V(\mf{e}(t)) = \mf{e}(t)^\top \mf{P}(t) \mf{e}(t)$ are valid Lyapunov functions.
\end{remark}

\subsection{Auxiliary Results}

We summarize some results for the convergence of equations \eqref{eq:dkb-p} and \eqref{eq:dkb-z1}-\eqref{eq:dkb-z2} in the ODEFTC algorithm, which are relevant to prove its estimation performance in the LTV case. Their proofs are given in \cite{odeftc}.

\begin{lemma}\label{lem:Z} {Let Assumption \ref{assum:bounds-Z} hold. Consider a graph $\mathcal{G}$ with algebraic connectivity $\lambda_\mathcal{G}$, $N$ nodes and $\ell$ edges. 
For given design parameters $\alpha > 0$ and $\gamma \in (0,1)$, set $\xi \geq 2L/(\alpha \sqrt{\lambda_{\mathcal{G}}})$ and $T_{\max}=\ell \pi/(\alpha \gamma \lambda_{\mathcal{G}})$. Then, the consensus protocol \eqref{eq:dkb-z1}-\eqref{eq:dkb-z2} fulfills $\hat{\mf{Z}}_i(t) = \bar{\mf{Z}}(t) \ \forall t \geq T_{\max}$}, where $\bar{\mf{Z}}(t) = \mf{C}(t)^\top\mf{R}(t)^{-1}\mf{C}(t)$.
\end{lemma}

\begin{theorem}\label{th:p-inf}
Let Assumptions \ref{assum:observable}, \ref{assum:bounds} and \ref{assum:bounds-Z} hold along with the conditions from Lemma \ref{lem:Z}. The estimated error covariance matrix $\mf{P}_i(t)$ given by \eqref{eq:dkb-p} complies with the following properties:
\begin{enumerate}
    \item \label{it:p-bounded} $\mf{P}_i(t)$ is uniformly bounded $\forall t\geq 0$.
    \item \label{it:p-after-tmax} {After a fixed time $T_{\max}$, $\mf{P}_i(t)$ evolves as \eqref{eq:ckb-p} from the optimal centralized filter.}
    \item \label{it:p-lim-ckb} $\lim_{t\to\infty}(\mf{P}_i(t) - {\mf{P}}(t))=\mf{0}$  with $\mf{P}(t)$ from the optimal centralized filter \eqref{eq:ckb-p}.
\end{enumerate}
\end{theorem}

\begin{remark}
In the centralized filter, the expression for $\mf{P}(t)$ in \eqref{eq:ckb-p} is obtained as the error covariance of the estimates \eqref{eq:ckb-x}. In contrast, in the ODEFTC algorithm, \eqref{eq:dkb-x} and \eqref{eq:dkb-p} have been designed independently from each other, meaning that we need to differentiate between the \emph{estimated} error covariance $\mf{P}_i(t)$ given by \eqref{eq:dkb-p} and the \emph{true} error covariance $\bm{\mathcal{P}}_i(t) = \cov \{\mf{x}(t) - \hat{\mf{x}}_i(t)\}$ of the estimates from \eqref{eq:dkb-x}. Therefore, in order to show that ODEFTC recovers the performance of the centralized filter, we need to show that $\mf{P}_i(t) \rightarrow \mf{P}(t)$ and $\bm{\mathcal{P}}_i(t) \rightarrow \mf{P}(t)$ separately. The first result is already given in Theorem \ref{th:p-inf}, and the second will be provided in Theorem \ref{th:main}.
\end{remark}

\section{Optimality of ODEFTC for LTV Systems}

The remainder of this work is devoted to analyzing the estimates \eqref{eq:dkb-x} from the ODEFTC algorithm, in order to show that ODEFTC asymptotically recovers the performance of the optimal centralized filter in each node as the consensus gain $\kappa$ is increased, even in the case of LTV systems. Formally, we aim to prove the following theorem, which generalizes \cite[Theorem 2]{odeftc} to the LTV case. 

\begin{theorem}\label{th:main}
    Let Assumptions \ref{assum:observable}, \ref{assum:bounds} and \ref{assum:bounds-Z} hold, along with the conditions from Lemma \ref{lem:Z}. Set $\kappa_0 = c^2/(2 r_1 \lambda_\mathcal{G})$, with $c$ and $r_1$ defined in Assumption \ref{assum:bounds} and $\lambda_{\mathcal{G}}$ being the algebraic connectivity of the graph. For $\kappa > \kappa_0$, the ODEFTC algorithm given by \eqref{eq:dkb} complies with the following properties for the estimates $\hat{\mf{x}}_i(t)$ and the true error covariance, $\bm{\mathcal{P}}_i(t) = \cov \{ \mf{x}(t) - \hat{\mf{x}}_i(t) \}$:
    \begin{enumerate}
        \item \label{it:unbiased} {The estimates $\hat{\mf{x}}_i(t)$ are unbiased.}
        \item \label{it:bounded} $\bm{\mathcal{P}}_i(t)$ is uniformly bounded $\forall t \geq 0$. 
        \item \label{it:opt-ltv} There exists a bound $b(\kappa)>0$ such that  $\|\lim_{t\to\infty}(\bm{\mathcal{P}}_i(t) - \mf{P}(t))\|\leq b(\kappa) $, where $\mf{P}(t)$ is the error covariance matrix from the centralized Kalman-Bucy filter \eqref{eq:ckb-p}, and $\lim_{\kappa\to\infty}b(\kappa) = 0$.
    \end{enumerate}
\end{theorem}

\begin{remark} Notice that Theorem \ref{th:main} also provides a sufficient lower bound $\kappa_0$ for the consensus gain to ensure stability, depending only on the bounds of the matrices $\mf{C}(t), \, \mf{R}(t)$ in the measurement model and the topology of the graph $\mathcal{G}$, which can also be computed in a distributed fashion, e.g. \cite{Montijano2017}. 
\end{remark}

\begin{remark}
    While the result of item \ref{it:opt-ltv} in Theorem \ref{th:main}, i.e., achieving a covariance arbitrarily close to the optimal one even for large networks, might appear surprising, it is theoretically possible since, with continuous-time communication, an arbitrary number of messages can be transmitted within any time interval. As usual, in a discretized implementation, a slight performance degradation is expected, as up to $N$ messages may be required in the worst case for information to propagate throughout the entire sensor network. Additionally, as $\kappa$ increases, smaller discretization steps are needed to maintain the stability of the discretized estimator.\end{remark}

In the subsequent sections, we include some necessary analyses before proving Theorem \ref{th:main} in Section \ref{sec:proof-th}.

\subsection{Error and Covariance Dynamics}
We begin our analysis by writing the estimation error dynamics in the sensor network. Let $\mf{e}_i(t) = \mf{x}(t) - \hat{\mf{x}}_i(t)$ be the error for the estimate computed by node $i$. Then, considering \eqref{eq:sys}, \eqref{eq:meas} and \eqref{eq:dkb-x}, its error dynamics are given by
\begin{equation}
\begin{aligned}
&\dot{\mf{e}}_i(t) = \dot{\mf{x}}(t) - \dot{\hat{\mf{x}}}_i(t) \\
&= (\mf{A}(t)-\mf{K}_i(t)\mf{C}_i(t)) \mf{e}_i(t)  + \mf{w}(t) - \mf{K}_i(t) \mf{v}_i(t) - \kappa \mf{P}_i(t) \textstyle \sum_{j \in \mathcal{N}_i} \left(\mf{e}_i(t) - \mf{e}_j(t) \right).
\end{aligned}
\end{equation}
Define the following matrices
\begin{equation}\label{eq:a-ast}
\begin{aligned}
    \mf{A}_i(t) &= \mf{A}(t) - \mf{K}_i(t) \mf{C}_i(t), \\
    \mf{A}^\ast(\kappa, t) &= \diag_{i=1}^N (\mf{A}_i(t)) - \kappa \, \diag_{i=1}^N(\mf{P}_i(t))(\mf{Q}_\mathcal{G} \otimes \mf{I}_n),
\end{aligned}
\end{equation}
recalling that $\mf{Q}_\mathcal{G}$ is the Laplacian matrix for the communication graph. 
We stack the errors as $\mf{e}(t) = \col_{i=1}^N(\mf{e}_i(t))$ and define the aggregate error covariance $\bm{\mathcal{P}}(t) = \cov\{\mf{e}(t)\}$, which has the individual error covariances $\bm{\mathcal{P}}_i(t)$ of the nodes as block diagonal elements, with the other elements being correlation terms due to all nodes observing the same noisy process \eqref{eq:sys}. Note that proving item \ref{it:opt-ltv} is equivalent to proving that $\bm{\mathcal{P}}(t) \rightarrow \mf{U}_N \otimes \mf{P}(t)$, i.e. that the aggregate behavior of the nodes is the same as $N$ centralized Kalman-Bucy filters using the same information.
Finally, defining $\mf{n}(t) = \mathds{1}_N \otimes \mf{w}(t) - \diag_{i=1}^N(\mf{K}_i(t))\mf{v}(t)$, we can write the dynamics for the aggregate error and its error covariance:
\begin{equation}\label{eq:error-dyn}
\begin{aligned}
\dot{\mf{e}}(t) = \mf{A}^\ast(\kappa, t) \mf{e}(t) + \mf{n}(t),
\end{aligned}
\end{equation}
\begin{equation}\label{eq:cov-dyn}
\begin{aligned}
\dot{\bm{\mathcal{P}}}(t) =& \mf{A}^\ast(\kappa, t) \bm{\mathcal{P}}(t) + \bm{\mathcal{P}}(t) \mf{A}^\ast(\kappa,t)^\top + \mf{U}_N \otimes \mf{W}(t)\\&\quad + \diag_{i=1}^N(\mf{K}_i(t)) \mf{R}(t) \diag_{i=1}^N(\mf{K}_i(t))^\top.
\end{aligned}
\end{equation}

\begin{remark}\label{rem:pi-p}
The dynamic matrix $\mf{A}^\ast(\kappa, t)$ defined in \eqref{eq:a-ast} depends on the local matrices $\mf{P}_i(t)$ computed as \eqref{eq:dkb-p} in the ODEFTC algorithm. By item \ref{it:p-lim-ckb} of Theorem \ref{th:p-inf}, we have that $\mf{P}_i(t) \rightarrow \mf{P}(t)$ asymptotically, where $\mf{P}(t)$ is the error covariance matrix from the centralized filter. To simplify the analysis, we first prove some relevant lemmas under the assumption of consensus on $\mf{P}_i(t) = \mf{P}(t)$. Then, we formally show that the obtained results apply despite the transitory in convergence of $\mf{P}_i(t)$, and we use them to prove Theorem \ref{th:main}.
\end{remark}

\subsection{Choice of Consensus Gain for Stability}

From \eqref{eq:error-dyn} and \eqref{eq:cov-dyn}, it is evident that the stability of the error and the error covariance dynamics depends on the choice of the consensus gain $\kappa$. The following lemma states the existence of an appropriate choice of $\kappa$ to guarantee stability.

\begin{lemma}\label{lem:stable} 
    Let Assumptions \ref{assum:observable} and \ref{assum:bounds} hold. Moreover, assume that $\mf{P}_i(t) = \mf{P}(t), \, \forall i \in \mathcal{V}, \, \forall t \geq 0,$ with $\mf{P}(t)$ from \eqref{eq:ckb-p}. Then, there exists $\kappa_0>0$ such that $\forall \kappa > \kappa_0$ the origin of the system $\dot{\mf{e}}(t) = \mf{A}^\ast(\kappa,t) \mf{e}(t)$, with $\mf{A}^\ast(\kappa, t)$ defined in \eqref{eq:a-ast}, is uniformly asymptotically stable. 
\end{lemma}
\begin{proof}

We separate $\mf{e}(t) = \bar{\mf{e}}(t) + \tilde{\mf{e}}(t)$, with $\bar{\mf{e}}(t) = \mathds{1}_N \otimes \bm{\upvarepsilon}(t)$ where $\bm{\upvarepsilon}(t) = (\mathds{1}^\top_N/N \otimes \mf{I}_n)\mf{e}(t)$, and $\tilde{\mf{e}}(t) = (\mf{H} \otimes \mf{I}_n) \mf{e}(t)$ with $\mf{H}=\mf{I}_N - (1/N)\mf{U}_N$. This is, we have divided $\mf{e}(t)$ into a component $\bar{\mf{e}}(t)$, which is parallel to the consensus direction, and a disagreement component $\tilde{\mf{e}}(t)$, which is orthogonal to the consensus direction.
In order to show that $\dot{\mf{e}}(t) = \mf{A}^\ast(\kappa,t) \mf{e}(t)$ is asymptotically stable towards the origin for a large enough $\kappa$, we analyze each of the components $\bar{\mf{e}}(t)$ and $\tilde{\mf{e}}(t)$.

First, we consider $\bar{\mf{e}}(t)$, whose dynamics are given by
\begin{equation}\label{eq:dyn-ebar}
\begin{aligned}
    &\dot{\bar{\mf{e}}}(t) = \mathds{1}_N \otimes (\mathds{1}^\top_N/N \otimes \mf{I}_n) \dot{\mf{e}}(t)\\
    &= \mathds{1}_N \otimes (\mathds{1}^\top_N/N \otimes \mf{I}_n) (\diag_{i=1}^N(\mf{A}_i(t)) - \kappa(\mf{Q}_\mathcal{G} \otimes \mf{P}(t)))\mf{e}(t) \\
    &= \mathds{1}_N \otimes (\mathds{1}^\top_N/N \otimes \mf{I}_n) (\mf{I}_N \otimes \mf{A}(t)) (\mathds{1}_N \otimes \bm{\upvarepsilon}(t))\\
    &\quad - \mathds{1}_N \otimes (\mathds{1}^\top_N/N \otimes \mf{I}_n) \diag_{i=1}^N(\mf{K}_i(t)\mf{C}_i(t)) (\mathds{1}_N \otimes \bm{\upvarepsilon}(t))
    ,
\end{aligned}
\end{equation}
where we have used the fact that $\mathds{1}^\top_N \mf{Q}_\mathcal{G} = \mf{0}$ and $\mathds{1}^\top_N \mf{e}(t) = \mathds{1}^\top_N(\bar{\mf{e}}(t) + \tilde{\mf{e}}(t)) = \mathds{1}^\top_N\bar{\mf{e}}(t)$, since $\mathds{1}^\top\mf{H}=\mf{0}$.

Note that
\begin{equation}
    (\mathds{1}^\top_N/N \otimes \mf{I}_n) (\mf{I}_N \otimes \mf{A}(t)) (\mathds{1}_N \otimes \bm{\upvarepsilon}(t)) = \mf{A}(t) \bm{\upvarepsilon}(t).
\end{equation}
In addition, it can be verified that
\begin{equation}
\begin{aligned}
    &(\mathds{1}^\top_N/N \otimes \mf{I}_n) \diag_{i=1}^N(\mf{K}_i(t)\mf{C}_i(t)) (\mathds{1}_N \otimes \bm{\upvarepsilon}(t)) \\
    &= (1/N) \textstyle \sum_{i=1}^N \left( N \mf{P}(t) \mf{C}_i(t)^\top \mf{R}_i(t)^{-1} \mf{C}_i(t) \bm{\upvarepsilon}(t) \right) \\
    &= \mf{P}(t) \textstyle \sum_{i=1}^N \left( \mf{C}_i(t)^\top \mf{R}_i(t)^{-1} \mf{C}_i(t) \right) \bm{\upvarepsilon}(t) \\
    &= \mf{P}(t) \mf{C}(t)^\top \mf{R}(t)^{-1} \mf{C}(t) \bm{\upvarepsilon}(t) = \mf{K}(t) \mf{C}(t) \bm{\upvarepsilon}(t).
\end{aligned}
\end{equation}

With these considerations, we can write 
\begin{equation}\label{eq:dyn-ebar}
    \dot{\bar{\mf{e}}}(t) = \mathds{1}_N \otimes (\mf{A}(t) - \mf{K}(t)\mf{C}(t))\bm{\upvarepsilon}(t) = (\mf{I}_N \otimes \mf{A}_c(t)) \bar{\mf{e}}(t),
\end{equation}
where $\mf{A}_c(t) = \mf{A}(t) - \mf{K}(t)\mf{C}(t)$ is the dynamic matrix for the estimation error in the centralized Kalman-Bucy filter. Since the centralized filter is known to be asymptotically stable \cite[Theorem 4]{Kalman1961}, it follows that $\bar{\mf{e}}(t)$ converges to the origin asymptotically, regardless of the choice of $\kappa$.

Now, we address the disagreement term $\tilde{\mf{e}}(t)$, which has dynamics $$
    \dot{\tilde{\mf{e}}}(t) = \mf{A}^\ast(\kappa,t) \tilde{\mf{e}}(t),
$$ showing that it is stable for an appropriate choice of $\kappa$.
We use the following Lyapunov function
\begin{equation}
    V(\tilde{\mf{e}}(t)) = \tilde{\mf{e}}(t)^\top (\mf{I}_N \otimes \mf{P}(t)^{-1}) \tilde{\mf{e}}(t),  
\end{equation}
which yields
\begin{equation}\label{eq:lyap-disagr}
\begin{aligned}
    \dot{V}(\tilde{\mf{e}}(t)) &= \tilde{\mf{e}}(t)^\top \mf{A}^\ast(\kappa,t)^\top (\mf{I}_N \otimes \mf{P}(t)^{-1}) \tilde{\mf{e}}(t)  + \tilde{\mf{e}}(t)^\top (\mf{I}_N \otimes \mf{P}(t)^{-1}) \mf{A}^\ast(\kappa,t) \tilde{\mf{e}}(t) \\&\quad + \tilde{\mf{e}}(t)^\top (\mf{I}_N \otimes \textstyle\frac{\text{d}}{\text{d}t}(\mf{P}(t)^{-1})) \tilde{\mf{e}}(t) \\
    &=\tilde{\mf{e}}(t)^\top \diag_{i=1}^N(\mf{P}(t)^{-1} \mf{A}_{i}(t) + \mf{A}_{i}(t)^\top \mf{P}(t)^{-1} \\&\quad+ \textstyle \frac{\text{d}}{\text{d}t}({\mf{P}}(t)^{-1})) \tilde{\mf{e}}(t) - 2\kappa \tilde{\mf{e}}(t)^\top (\mf{Q}_\mathcal{G} \otimes \mf{I}_N) \tilde{\mf{e}}(t). \\
\end{aligned}
\end{equation}
Using the identity $\frac{\text{d}}{\text{d}t}({\mf{P}}(t)^{-1}) = - \mf{P}(t)^{-1} \dot{\mf{P}}(t) \mf{P}(t)^{-1}$ along with \eqref{eq:ckb-p}, we can write
\begin{equation}\label{eq:deriv-p}
\begin{aligned}
\textstyle \frac{\text{d}}{\text{d}t}({\mf{P}}(t)^{-1}) &= - \mf{P}(t)^{-1}\mf{A}(t) - \mf{A}(t)^\top\mf{P}(t)^{-1} \\&\quad - \mf{P}(t)^{-1} \mf{W}(t) \mf{P}(t)^{-1} + \mf{C}(t)^\top \mf{R}(t)^{-1} \mf{C}(t).
\end{aligned}
\end{equation}
Noticing that
\begin{equation}
\begin{aligned}
    &\mf{P}(t)^{-1} \mf{A}_i(t) = 
    \mf{P}(t)^{-1} \mf{A}(t) - N \mf{C}_i(t)^\top\mf{R}_i(t)^{-1} \mf{C}_i(t),
\end{aligned}
\end{equation}
it follows that
\begin{equation}
\begin{aligned}
&\mf{P}(t)^{-1} \mf{A}_i(t) + \mf{A}_i(t)^\top \mf{P}(t)^{-1} + \textstyle \frac{\text{d}}{\text{d}t}({\mf{P}}(t)^{-1}) \\
&= \mf{C}(t)^\top \mf{R}(t)^{-1} \mf{C}(t) - \mf{P}(t)^{-1} \mf{W}(t)  \mf{P}(t)^{-1} - 2N\mf{C}_i(t)^\top\mf{R}_i(t)^{-1} \mf{C}_i(t) \\
&\prec \mf{C}(t)^\top \mf{R}(t)^{-1} \mf{C}(t),
\end{aligned}
\end{equation}
due to the positive semidefiniteness of $\mf{P}(t)^{-1} \mf{W}(t)  \mf{P}(t)^{-1}$ and $\mf{C}_i(t)^\top\mf{R}_i(t)^{-1} \mf{C}_i(t)$.
Substituting this in \eqref{eq:lyap-disagr}, we reach
\begin{equation}\label{eq:lyap-disagr3}
\begin{aligned}
    \dot{V}(\tilde{\mf{e}}(t)) &\leq \tilde{\mf{e}}(t)^\top (\mf{I}_N \otimes \mf{C}(t)^\top \mf{R}(t)^{-1} \mf{C}(t))\tilde{\mf{e}}(t)  -2 \kappa \tilde{\mf{e}}(t)^\top(\mf{Q}_\mathcal{G} \otimes \mf{I}_N) \tilde{\mf{e}}(t) \\
    &\leq \|\mf{C}(t)^\top \mf{R}(t)^{-1} \mf{C}(t)\| \|\tilde{\mf{e}}(t)\|^2 - 2 \kappa \lambda_\mathcal{G} \|\tilde{\mf{e}}(t)\|^2, \\
    &\leq (c^2/r_1) \|\tilde{\mf{e}}(t)\|^2 - 2 \kappa \lambda_\mathcal{G} \|\tilde{\mf{e}}(t)\|^2,
\end{aligned}
\end{equation}
where we have used the fact that
$$
    \|\mf{C}(t)^\top \mf{R}(t)^{-1} \mf{C}(t)\| \leq \|\mf{C}(t)\|^2 \|\mf{R}(t)^{-1}\| \leq c^2/r_1,
$$
with the constants $c$ and $r_1$ defined as the bounds for the matrices in Assumption \ref{assum:bounds}. 
From this, we conclude that if 
\begin{equation}
     \kappa > \frac{c^2}{2 r_1 \lambda_\mathcal{G}} =: \kappa_0
\end{equation}
then, all trajectories of $\tilde{\mf{e}}(t)$, and therefore of the original vector $\mf{e}(t)$ as well, converge to the origin asymptotically, completing the proof.
\end{proof}

In light of the proof of Lemma \ref{lem:stable}, the following result for an appropriate choice of $\kappa$ is obtained straightforwardly.

\begin{corollary}\label{cor:kappa}
Let the assumptions in Lemma \ref{lem:stable} hold. A sufficient choice of $\kappa$ to stabilize the system $\dot{\mf{e}}(t) = \mf{A}^\ast(\kappa,t)\mf{e}(t)$ is given by $\kappa > c^2/(2 r_1 \lambda_\mathcal{G}),$ with $c$ and $r_1$ defined in Assumption \ref{assum:bounds} and $\lambda_\mathcal{G}$ being the algebraic connectivity of the graph. 
\end{corollary}

\subsection{Analysis under Consensus Assumption for $\mf{P}_i(t)$}

Recall our goal to prove item \ref{it:opt-ltv} of Theorem \ref{th:main}. Taking the aggregate covariance $\bm{\mathcal{P}}(t)$ with dynamics \eqref{eq:cov-dyn}, we aim to verify that $\bm{\mathcal{P}}(t) \rightarrow \mf{U}_N \otimes \mf{P}(t)$ as $t \rightarrow \infty$, $\kappa\rightarrow \infty$. This is, we want to show that the aggregate error covariance for the estimates of ODEFTC behaves as that of $N$ centralized Kalman-Bucy filters using the global measurement information.
Therefore, we start by defining the covariance mismatch between ODEFTC and $N$ centralized filters.

Recalling that $\mf{A}_c(t) = \mf{A}(t) - \mf{K}(t) \mf{C}(t)$ and \eqref{eq:ckb-p}, we can write, by adding and subtracting $\mf{K}(t)\mf{R}(t)\mf{K}(t)(t)^\top$,
\begin{equation}\label{eq:cov-p-central}
\begin{aligned}
    \dot{\mf{P}}(t) = \mf{A}_c(t) \mf{P}(t) + \mf{P}(t) \mf{A}_c(t)^\top+\mf{W}(t) + \mf{K}(t)\mf{R}(t)\mf{K}(t)^\top
\end{aligned}
\end{equation}
Noting that $\kappa (\mf{Q}_\mathcal{G} \otimes \mf{P}(t))(\mf{U}_N \otimes \mf{P}(t)) =\mf{0}$ due to $\mf{Q}_\mathcal{G} \mf{U}_N = \mf{0}$ and considering \eqref{eq:cov-p-central}, we can write the following error covariance dynamics for $N$ centralized filters:
\begin{equation}
\begin{aligned}
    (\mf{U}_N \otimes \dot{\mf{P}}(t)) &= (\diag_{i=1}^N(\mf{A}_c(t)) - \kappa (\mf{Q}_\mathcal{G} \otimes \mf{P}(t)))(\mf{U}_N \otimes \mf{P}(t)) \\ &\quad + (\mf{U}_N \otimes \mf{P}(t))(\diag_{i=1}^N(\mf{A}_c(t)) - \kappa (\mf{Q}_\mathcal{G} \otimes \mf{P}(t)))^\top \\  &\quad + \mf{U}_N \otimes \mf{W}(t) + \diag_{i=1}^N(\mf{K}(t))(\mf{U}_N \otimes \mf{R}(t)) \diag_{i=1}^N(\mf{K}(t))^\top.
\end{aligned}
\end{equation}

With this, we define the covariance mismatch between ODEFTC and $N$ centralized filters as 
\begin{equation}\label{eq:mismatch-x}
\mf{X}(t) = \bm{\mathcal{P}}(t) - \mf{U}_N \otimes \mf{P}(t)
\end{equation}
and we write
\begin{equation}\label{eq:dyn-x}
    \dot{\mf{X}}(t) = \mf{A}^\ast(\kappa,t) \mf{X}(t) + \mf{X}(t) \mf{A}^\ast(\kappa, t)^\top + \bm{\Sigma}(t),
\end{equation}
where $\bm{\Sigma}(t)$ contains terms related to the differences between $\mf{A}_{i}(t)$ and $\mf{A}_c(t)$, as well as between $\mf{K}(t)$ and $\mf{K}_{i}(t)$. Particularly, let $\mf{G}_i(t) = \mf{C}_i(t)^\top \mf{R}_i(t)^{-1}(t) \mf{C}_i(t)$, $\mf{G}(t) = \mf{C}(t)^\top \mf{R}(t)^{-1}(t) \mf{C}(t) = \sum_{i=1}^N \mf{G}_i(t)$ and $\mf{G}_d(t) = \diag_{i=1}^N( \mf{G}_i(t))$, and with $\mf{P}_i(t) = \mf{P}(t)$, we have
\begin{equation}\label{eq:sigma}
\begin{aligned}
    \bm{\Sigma}(t) &= \diag_{i=1}^N(\mf{A}_i(t) - \mf{A}_c(t))(\mf{U}_N \otimes \mf{P}(t)) \\ &\quad + (\mf{U}_N \otimes \mf{P}(t))  \diag_{i=1}^N( \mf{A}_i(t) - \mf{A}_c(t))^\top \\&\quad + \diag_{i=1}^N(\mf{K}_i(t))\mf{R}(t)\diag_{i=1}^N(\mf{K}_i(t)) \\& \quad - \diag_{i=1}^N(\mf{K}(t))(\mf{U}_N \otimes \mf{R}(t))\diag_{i=1}^N(\mf{K}(t))\\
    &=  N^2(\mf{I}_N \otimes \mf{P}(t)) \mf{G}_d(t) (\mf{I}_N \otimes \mf{P}(t))  \\&\quad - N (\mf{U}_N \otimes \mf{P}(t)) \mf{G}_d(t) (\mf{I}_N \otimes \mf{P}(t))  \\&\quad- N(\mf{I}_N \otimes \mf{P}(t)) \mf{G}_d (\mf{U}_N \otimes \mf{P}(t)) + \mf{U}_N \otimes \mf{P}(t)\mf{G}(t)\mf{P}(t).
\end{aligned}
\end{equation}
In order to prove item \ref{it:opt-ltv} of Theorem \ref{th:main}, we need to analyze the convergence of $\mf{X}(t)$. Note that it has stable dynamics \eqref{eq:dyn-x} for an appropriate choice of $\kappa$ according to Lemma \ref{lem:stable}. Therefore, the presence of $\bm{\Sigma}(t)$, which is a disturbance caused by the difference of available measurement information at each node, is the only element preventing the convergence of $\mf{X}(t)$ to the origin. We will show that increasing the value of $\kappa$ serves to diminish the effect of this disturbance.

Before proving this, we derive some technical lemmas.

\begin{lemma}\label{lem:sigma-ort}
    The disturbance $\bm{\Sigma}(t)$ in \eqref{eq:dyn-x} fulfills $$(\mathds{1}_N^\top \otimes \mf{I}_n)\bm{\Sigma}(t)(\mathds{1}_N \otimes \mf{I}_n) = \mf{0}.$$
\end{lemma}
\begin{proof}
Recalling that $\bm{\Sigma}(t)$ was composed of several terms in \eqref{eq:sigma}, we analyze each term as follows. 
The first term yields
\begin{equation}
\begin{aligned}
    &N^2 (\mathds{1}^\top \otimes \mf{I}_N) (\mf{I}_N \otimes \mf{P}(t)) \mf{G}_d(t) (\mf{I}_N \otimes \mf{P}(t)) (\mathds{1}_N \otimes \mf{I}_N) \\& = N^2 (\mathds{1}_N^\top \otimes \mf{P}(t)) \mf{G}_d(t) (\mathds{1}_N \otimes \mf{P}(t)) \\&= N^2 \mf{P}(t) \textstyle \sum_{i=1}^N \mf{G}_i(t) = N^2\mf{P}(t) \mf{G}(t) \mf{P}(t).
\end{aligned}
\end{equation}
The second and third terms are symmetric. Analyzing one of them, we reach
\begin{equation}
\begin{aligned}
    &N (\mathds{1}^\top \otimes \mf{I}_N) (\mf{U}_N \otimes \mf{P}(t)) \mf{G}_d(t) (\mf{I}_N \otimes \mf{P}(t)) (\mathds{1}_N \otimes \mf{I}_N) \\& = N (\mathds{1}^\top \mf{U}_N \otimes \mf{P}(t))\mf{G}_d(t) (\mathds{1}_N \otimes \mf{P}(t)) \\&= N^2 (\mathds{1}_N^\top \otimes \mf{P}(t)) \mf{G}_d(t) (\mathds{1}_N \otimes \mf{P}(t)) = N^2 \mf{P}(t) \mf{G}(t) \mf{P}(t).
\end{aligned}
\end{equation}
For the last term, we have
\begin{equation}
\begin{aligned}
    &(\mathds{1}^\top \otimes \mf{I}_N) (\mf{U}_N \otimes \mf{P}(t)\mf{G}(t)\mf{P}(t)) (\mathds{1}_N \otimes \mf{I}_N) \\& = \mathds{1}_N^\top \mf{U}_N \mathds{1}_N \otimes \mf{P}(t) \mf{G}(t) \mf{P}(t) = N^2 \mf{P}(t) \mf{G}(t) \mf{P}(t),
\end{aligned}
\end{equation}
noting that $\mathds{1}_N^\top \mf{U}_N \mathds{1}_N = \mathds{1}_N^\top \mathds{1}_N \mathds{1}_N^\top \mathds{1}_N = N^2 $. 
Considering the terms above and \eqref{eq:sigma}, the claim in the statement of the lemma follows. In addition, due to symmetry, $(\mathds{1}^\top_N \otimes \mf{I}_n)\bm{\Sigma}(t) = \mf{0}$ and $\bm{\Sigma}(t)(\mathds{1}_N \otimes \mf{I}_n) = \mf{0}$ also hold.
\end{proof}

\begin{lemma}\label{lem:ineq-kappa}
    Let the conditions and notation from Lemma \ref{lem:stable} hold. Set $\kappa > \kappa_0$, and separate it into $\kappa = \kappa_1 + \tilde{\kappa}$, with $\kappa > \kappa_1 > \kappa_0$. Then, the following inequality holds
\begin{equation}\label{eq:ineq}
    \dot{\mf{N}}(t) + \mf{NA}^\ast(\kappa,t) + \mf{A}^\ast(\kappa,t) ^\top \mf{N}(t) \preceq - \mf{L},
\end{equation}
with $\mf{N}(t) = \mf{I}_N \otimes \mf{P}(t)^{-1}$, $\mf{L} = 2\tilde{\kappa}\lambda_{\mathcal{G}}(\mf{H} \otimes \mf{I}_n) \succeq \mf{0}$ and $\mf{H} = \mf{I}_N - \mf{U}_N/N$.
\end{lemma}
\begin{proof}
    From the proof of Lemma \ref{lem:stable}, recall the system $\dot{\mf{e}}(t) = \mf{A}^\ast(\kappa, t)\mf{e}(t)$, where we had separated $\mf{e}(t) = \bar{\mf{e}}(t) + \tilde{\mf{e}}(t)$. Note that we have the following Lyapunov function
\begin{equation}\label{eq:lyap-combi}
\begin{aligned}
    V(\mf{e}(t)) &= \mf{e}(t)^\top (\mf{I}_N \otimes \mf{P}(t)^{-1})\mf{e}(t) \\
    &= \bar{\mf{e}}(t)^\top (\mf{I}_N \otimes \mf{P}(t)^{-1})\bar{\mf{e}}(t) + \tilde{\mf{e}}(t)^\top (\mf{I}_N \otimes \mf{P}(t)^{-1})\tilde{\mf{e}}(t) \\
    &= V(\bar{\mf{e}}(t)) + V(\tilde{\mf{e}}(t)),
\end{aligned}
\end{equation}
since the cross terms $\bar{\mf{e}}(t)(\mf{I}_N \otimes \mf{P}(t)^{-1})\tilde{\mf{e}}(t) = 0$ due to $\bar{\mf{e}}(t)$ and $\tilde{\mf{e}}(t)$ being orthogonal. 

With the dynamics \eqref{eq:dyn-ebar}, we have
 \begin{equation}\label{eq:lyap-cons2}
\begin{aligned}
\dot{V}(\bar{\mf{e}}(t)) = \bar{\mf{e}}(t)^\top (\mf{I}_N \otimes (\mf{P}(t)^{-1} \mf{A}_c(t) + \mf{A}_c(t)^\top \mf{P}(t)^{-1} + \textstyle \frac{\text{d}}{\text{d}t}({\mf{P}}(t)^{-1}))) \bar{\mf{e}}(t).
\end{aligned}
\end{equation}
With similar manipulations as in the proof of Lemma \ref{lem:stable}, it can be verified that
\begin{equation}\label{eq:p-inverse}
\begin{aligned}
    &\mf{P}(t)^{-1} \mf{A}_c(t) + \mf{A}_c(t)^\top \mf{P}(t)^{-1} + \textstyle \frac{\text{d}}{\text{d}t}({\mf{P}}(t)^{-1}) \\&= -\mf{P}(t)^{-1}\mf{W}(t)\mf{P}(t)^{-1} -\mf{C}(t)^\top\mf{R}(t)^{-1}\mf{C}(t) \preceq \mf{0},
\end{aligned}
\end{equation}
noting that $\mf{P}(t)^{-1}\mf{W}(t)\mf{P}(t)^{-1}$ and $\mf{C}(t)^\top\mf{R}(t)^{-1}\mf{C}(t)$ are positive semidefinite.
Therefore, \eqref{eq:lyap-cons2} becomes
\begin{equation}\label{lyap:cons-3}
\begin{aligned}
\dot{V}(\bar{\mf{e}}(t)) = -\bar{\mf{e}}(t)^\top (\mf{I}_N \otimes (\mf{P}(t)^{-1}\mf{W}(t)\mf{P}(t)^{-1}   + \mf{C}(t)^\top\mf{R}(t)^{-1} \mf{C}(t))) \bar{\mf{e}}(t) \leq 0.
\end{aligned}
\end{equation}

For the term $\tilde{\mf{e}}(t)$, recall that we had \eqref{eq:lyap-disagr3}. With the separation $\kappa = \kappa_1 + \tilde{\kappa}$, we can further write
\begin{equation}\label{eq:lyap-disagr-4}
\begin{aligned}
    \dot{V}(\tilde{\mf{e}}(t)) &\leq (c^2/r_1 - 2\kappa_1\lambda_{\mathcal{G}}) \| \tilde{\mf{e}}(t)\|^2 - 2\tilde{\kappa} \lambda_{\mathcal{G}}\| \tilde{\mf{e}}(t)\|^2 \\
    &\leq - 2\tilde{\kappa} \lambda_{\mathcal{G}}\| \tilde{\mf{e}}(t)\|^2,
\end{aligned}
\end{equation}
thanks to the choice $\kappa_1 > \kappa_0$.
With \eqref{eq:lyap-combi}, \eqref{lyap:cons-3} and \eqref{eq:lyap-disagr-4} it holds that
\begin{equation}\label{eq:lyap1}
    \dot{V}(\mf{e}(t)) \leq -2\tilde{\kappa} \lambda_{\mathcal{G}} \| \tilde{\mf{e}}(t) \|^2 = -2 \tilde{\kappa} \lambda_{\mathcal{G}} \mf{e}(t)^\top (\mf{H} \otimes \mf{I}_n) \mf{e}(t),
\end{equation}
noting that $\mf{H}^2 = \mf{H}$. In addition, considering the dynamics for $\mf{e}(t)$, the following also holds:
\begin{equation}\label{eq:lyap2}
    \dot{V}(\mf{e}(t)) = \mf{e}(t)^\top(\dot{\mf{N}}(t) + \mf{NA}^\ast(\kappa,t) + \mf{A}^\ast(\kappa,t) ^\top \mf{N}(t)) \mf{e}(t),
\end{equation}
where $\mf{N}(t) = \mf{I}_N \otimes \mf{P}(t)^{-1}$. Therefore, comparing \eqref{eq:lyap1} and \eqref{eq:lyap2}, the result in \eqref{eq:ineq} follows.
\end{proof}

Next, we can show that Theorem \ref{th:main} holds under the simplyfing assumption, recalling Remark \ref{rem:pi-p}.

\begin{lemma}\label{lem:proof-with-p}
    Let all conditions from Theorem \ref{th:main} hold, and assume that $\mf{P}_i(t) = \mf{P}(t), \, \forall i\in\mathcal{V}, \, \forall t \geq 0$, with $\mf{P}(t)$ from \eqref{eq:ckb-p}. Then, the estimates $\hat{\mf{x}}_i(t)$ from \eqref{eq:dkb-x} fulfill the properties stated in Theorem \ref{th:main}.
\end{lemma}

\begin{proof}
Recall that the estimation error for the network is given by \eqref{eq:error-dyn}, and the error covariance by \eqref{eq:cov-dyn}. According to Lemma \ref{lem:stable}, the choice of $\kappa > \kappa_0$ guarantees their stability. 

Item \ref{it:unbiased} of Theorem \ref{th:main} follows from the initialization $\hat{\mf{x}}_i(0)= \mf{x}_0$, from which we have $\mathbb{E}\{\mf{e}(0)\}=\mf{0}$, along with $\mathbb{E}\{\mf{n}(t)\}=\mf{0}$ by virtue of $\mf{w}(t)$ and $\mf{v}(t)$ being zero-mean disturbances. Considering \eqref{eq:error-dyn}, together they imply $\mathbb{E}\{\mf{e}(t)\}=\mf{0}, \ \forall t\geq 0$.

Item \ref{it:bounded} holds due to the stable dynamics \eqref{eq:cov-dyn} from the choice of $\kappa$, and the fact that $\mf{U}_N \otimes \mf{W}(t) + \diag_{i=1}^N(\mf{K}_i(t)) \mf{R}(t) \diag_{i=1}^N(\mf{K}_i(t))^\top$ is uniformly bounded, recalling the uniform bounds for the matrices in Assumption \ref{assum:bounds} and noting that $\mf{K}_i(t)$ is bounded due to $\mf{P}_i(t) = \mf{P}(t)$ being bounded (see Remark \ref{rem:p-ckb}).

In the remainder, we prove the key result given in item \ref{it:opt-ltv} of Theorem \ref{th:main}, by analyzing the convergence of the covariance mismatch \eqref{eq:dyn-x}. Particularly, we need to show that the covariance mismatch $\mf{X}(t)$ from \eqref{eq:mismatch-x} fulfills $\mf{X}(t) \rightarrow \mf{0}$ as $\kappa \rightarrow \infty$.

For convenience, we will consider the vectorized version of \eqref{eq:dyn-x}, which is given by
\begin{equation}\label{eq:vect-X}
    \dot{\bm{\uprho}}(t) = \mf{A}_{\bm{\uprho}}(\kappa,t) \bm{\uprho}(t) + \bm{\upsigma}(t),
\end{equation}
where $\bm{\uprho}(t) = \vect (\bm{\mathcal{P}}(t))$, $\bm{\upsigma}(t) = \vect(\bm{\Sigma}(t))$ and
\begin{equation}
    \mf{A}_{\bm{\uprho}}(\kappa,t) = \mf{I}_N \otimes \mf{A}^\ast(\kappa, t) + \mf{A}^{\ast}(\kappa, t) \otimes \mf{I}_N.
\end{equation}

Consider the following Lyapunov function, 
\begin{equation}
    V(\bm{\uprho}(t)) = \bm{\uprho}(t)^\top(\mf{N}(t) \otimes \mf{N}(t))\bm{\uprho}(t),
\end{equation}
with $\mf{N}(t) = \mf{I}_N \otimes \mf{P}(t)^{-1}$. Taking the derivative, we reach
\begin{equation}\label{eq:lyap-rho}
\begin{aligned}
    \dot{V}(\bm{\uprho}(t)) &= \bm{\uprho}(t)^\top(\dot{\mf{N}}(t) \otimes \mf{N}(t) + \mf{N}(t) \otimes \dot{\mf{N}}(t) \\ &\quad + \mf{A}_{\bm{\uprho}}(\kappa,t)^\top (\mf{N}(t) \otimes \mf{N}(t)) + (\mf{N}(t) \otimes \mf{N}(t))\mf{A}_{\bm{\uprho}}(\kappa,t))\bm{\uprho}(t) \\
    &\quad +2 \bm{\uprho}(t)^\top (\mf{N}(t) \otimes \mf{N}(t)) \bm{\sigma}(t).
\end{aligned}
\end{equation}
In addition, recalling Lemma \ref{lem:ineq-kappa}, we have that
\begin{equation}
\begin{aligned}
    &\dot{\mf{N}}(t) \otimes \mf{N}(t) + \mf{N}(t) \otimes \dot{\mf{N}}(t) \\
    &\preceq (-\mf{L} -\mf{N}(t)\mf{A}^\ast(\kappa,t) - \mf{A}^\ast(\kappa,t)^\top \mf{N}(t)) \otimes \mf{N}(t) \\
    &\quad + \mf{N}(t) \otimes (-\mf{L} -\mf{N}(t)\mf{A}^\ast(\kappa,t) - \mf{A}^\ast(\kappa,t)^\top \mf{N}(t)),
\end{aligned}
\end{equation}
and
\begin{equation}
\begin{aligned}
    &\mf{A}_{\bm{\uprho}}(\kappa,t)^\top (\mf{N}(t) \otimes \mf{N}(t)) + (\mf{N}(t) \otimes \mf{N}(t))\mf{A}_{\bm{\uprho}}(\kappa,t) \\
    &= \mf{N}(t) \otimes (\mf{A}^\ast(\kappa,t)^\top \mf{N}(t) + \mf{N}(t) \mf{A}^\ast(\kappa,t)) \\
    &\quad + (\mf{N}(t) \mf{A}^\ast(\kappa,t) +\mf{A}^\ast(\kappa,t)^\top \mf{N}(t)) \otimes \mf{N}(t).
\end{aligned}
\end{equation}
Substituting these results into \eqref{eq:lyap-rho}, it follows that
\begin{equation}\label{eq:lyap-rho2}
\begin{aligned}
    \dot{V}(\bm{\uprho}(t)) &\leq - \bm{\uprho}(t)^\top(\mf{N}(t) \otimes \mf{L} + \mf{L} \otimes \mf{N}(t))\bm{\uprho}(t) \\
    &\quad+ 2 \bm{\uprho}(t)^\top (\mf{N}(t) \otimes \mf{N}(t)) \bm{\sigma}(t) \\
    &\leq - 2 \tilde{\kappa}\lambda_{\mathcal{G}} \bm{\uprho}(t)^\top(\mf{N}(t) \otimes \mf{H} \otimes \mf{I}_n + \mf{H} \otimes \mf{I}_n \otimes \mf{N}(t))\bm{\uprho}(t) \\&\quad + 2 \bm{\uprho}(t)^\top (\mf{N}(t) \otimes \mf{N}(t)) \bm{\sigma}(t).
\end{aligned}
\end{equation}
Note that the disagreement projection matrix $\mf{H}$ appears now. 
We can divide $\bm{\uprho}(t)$ into a consensus component $\bar{\bm{\uprho}}(t) = (\mf{U}_N \otimes \mf{I}_{Nn^2}) \bm{\uprho}(t)$ and a disagreement component $\tilde{\bm{\uprho}}(t) = (\mf{H} \otimes \mf{I}_{Nn^2})$. Through algebraic manipulations detailed in the Appendix, using the property $\vect(\mf{A} \mf{B} \mf{C}) = (\mf{C}^\top \otimes \mf{A}) \vect (\mf{B})$ of the vectorization operator, and the property for $\bm{\Sigma}(t)$ from Lemma \ref{lem:sigma-ort}, it can be verified that the terms in \eqref{eq:lyap-rho2} simplify to
\begin{equation}
\begin{aligned}
&\bm{\uprho}(t)^\top(\mf{N}(t) \otimes \mf{H} \otimes \mf{I}_n + \mf{H} \otimes \mf{I}_n \otimes \mf{N}(t))\bm{\uprho}(t) \\&= \tilde{\bm{\uprho}}(t)^\top (\mf{N}(t) \otimes \mf{I}_{Nn} + \mf{I}_{Nn} \otimes \mf{N}(t)) \tilde{\bm{\uprho}}(t),
\end{aligned}
\end{equation}
\begin{equation}
\begin{aligned}
    &\bm{\uprho}(t)^\top (\mf{N}(t) \otimes \mf{N}(t)) \bm{\sigma}(t) = \tilde{\bm{\uprho}}(t)^\top (\mf{N}(t) \otimes \mf{N}(t)) \bm{\sigma}(t),
\end{aligned}
\end{equation}
where only the disagreement component $\tilde{\bm{\uprho}}(t)$ is involved.

Then, \eqref{eq:lyap-rho2} can be further evaluated as
\begin{equation}\label{eq:lyap-rho3}
\begin{aligned}
    \dot{V}(\bm{\uprho}(t)) &\leq - 2 \tilde{\kappa}\lambda_{\mathcal{G}} \tilde{\bm{\uprho}}(t)^\top(\mf{N}(t) \otimes \mf{I}_{Nn} + \mf{I}_{Nn} \otimes \mf{N}(t))\tilde{\bm{\uprho}}(t) \\&\quad + 2 \tilde{\bm{\uprho}}(t)^\top (\mf{N}(t) \otimes \mf{N}(t)) \bm{\sigma}(t) \\
    &\leq -4 \tilde{\kappa}\lambda_{\mathcal{G}} \lambda_{\min}(\mf{N}(t)) \| \tilde{\bm{\uprho}}(t) \|^2 + 2 \lambda_{\max}(\mf{N}(t) \otimes \mf{N}(t))\| \bm{\sigma}(t) \| \|\tilde{\bm{\uprho}}(t) \| \\
    &\leq -4 \tilde{\kappa}\lambda_{\mathcal{G}} (1/p_2) \| \tilde{\bm{\uprho}}(t) \|^2 + 2 (1/p_1^2) \bar{\sigma} \|\tilde{\bm{\uprho}}(t) \|
\end{aligned}
\end{equation}
where we have used that $\lambda_{\min}(\mf{N}(t)) = \lambda_{\min}(\mf{P}(t)^{-1}) \geq (1/p_2)$ and $\lambda_{\max}(\mf{N}(t) \otimes \mf{N}(t)) = \lambda_{\max}(\mf{N}(t))^2 = \lambda_{\max}(\mf{P}(t)^{-1})^2 = (1/p_1^2)$,  with $p_1, p_2$ defined in Remark \ref{rem:p-ckb}. Finally, we have used that $\| \bm{\sigma}(t) \| \leq \bar{\sigma}, \forall t \geq 0$, since $\bm{\Sigma}(t)$ is uniformly bounded due to the boundedness of the model matrices from Assumption \ref{assum:bounds} and the boundedness of $\mf{P}(t)$ from Remark \ref{rem:p-ckb}.  
From the above, $\dot{V}(\bm{\uprho}(t)) < 0$ if
\begin{equation}
        \| \tilde{\bm{\uprho}}(t) \| > \bar{\sigma} p_2/(2 \tilde{\kappa}\lambda_\mathcal{G} p_1^2) =: b'(\tilde{\kappa}).
\end{equation}
This is, solutions converge to an invariant region around the desired value, which becomes smaller as $\kappa$ is increased, with $\lim_{\tilde{\kappa} \to \infty} b'(\tilde{\kappa}) = 0$, which implies the result in item \ref{it:opt-ltv} of Theorem \ref{th:main} for an appropriate bound $b(\kappa)$.

\end{proof}

\subsection{Effect of the Transitory in Convergence of $\mf{P}_i(t)$}

So far, the results given in Lemmas \ref{lem:stable} and \ref{lem:proof-with-p} have assumed that the nodes operate with consensus on their matrices $\mf{P}_i(t)$, having $\mf{P}_i(t) = \mf{P}(t), \, \forall t\geq 0$. However, $\mf{P}_i(t)$ is computed by \eqref{eq:dkb-p} and the auxiliary consensus algorithm \eqref{eq:dkb-z1}-\eqref{eq:dkb-z2}. Therefore, even though convergence to $\mf{P}(t)$ is achieved asymptotically according to Theorem \ref{th:p-inf}, we must verify that the derived results hold for the actual $\mf{P}_i(t)$. The following lemma formally establishes this intuition.

\begin{lemma}\label{lem:comparison} Let the conditions from Theorem \ref{th:p-inf} hold. Consider the matrices defined in \eqref{eq:a-ast} with $\mf{P}_i(t)$ computed as \eqref{eq:dkb-p}, and define 
\begin{equation}
\begin{aligned}
\mf{K}_{\mf{P}, i}(t) &= N \mf{P}(t) \mf{C}_i(t)^\top \mf{R}_i(t)^{-1}, \\
\mf{A}_{\mf{P}, i}(t) &= \mf{A}(t) - \mf{K}_{\mf{P}, i}(t) \mf{C}_i(t), \\
\mf{A}^\ast_\mf{P}(\kappa,t ) &= \diag_{i=1}^N (\mf{A}_{\mf{P},i}(t)) - \kappa \, (\mf{Q}_\mathcal{G} \otimes \mf{P}(t)), \\
\end{aligned}
\end{equation}    
as their equivalents with $\mf{P}_i(t) = \mf{P}(t), \, \forall t\geq 0, \, \forall i \in \mathcal{V}$.
If the origin is a uniformly asymptotically stable equilibrium for the system $\dot{\mf{e}}(t) = \mf{A}^\ast_\mf{P}(\kappa,t) \mf{e}(t)$, then it is also for the system $\dot{\mf{e}}(t) = \mf{A}^\ast(\kappa,t)\mf{e}(t)$. Moreover, $\bm{\mathcal{P}}(t)$ from \eqref{eq:cov-dyn} asymptotically converges to the same solution as its counterpart with $\mf{A}^\ast_{\mf{P}}(\kappa,t)$ and $\mf{K}_{\mf{P},i}(t)$.
\end{lemma}
\begin{proof} Note that we can write
\begin{equation}
    \dot{\mf{e}}(t) = \mf{A}^\ast(\kappa,t) \mf{e}(t) = \mf{A}^\ast_\mf{P}(\kappa,t)\mf{e}(t) + \mf{g}(t), 
\end{equation}
which is now in terms of the ``nominal" system $ \dot{\mf{e}}(t) = \mf{A}^\ast_\mf{P}(\kappa,t)\mf{e}(t) $ and a ``perturbation" term due to the mismatch between $\mf{P}_i(t)$ and $\mf{P}(t)$, given by
 $   \mf{g}(t)  = \bm{\Psi}(t) \bar{\mf{e}}(t)$
with
\begin{equation}\label{eq:psi}
\begin{aligned}
    &\bm{\Psi}(t) = \mf{A}^\ast(\kappa, t) - \mf{A}^\ast_\mf{P}(\kappa,t) \\
    &= \diag_{i=1}^N(N (\mf{P}(t)-\mf{P}_i(t)) \mf{C}_i(t)^\top \mf{R}_i(t)^{-1} \mf{C}_i(t)) \\
    &\quad + \kappa \diag_{i=1}^N(\mf{P}(t) - \mf{P}_i(t))(\mf{Q}_\mathcal{G} \otimes \mf{I}_N).
\end{aligned}
\end{equation}
Note that 
 $   \| \mf{g}(t) \| \leq  \| \bm{\Psi}(t) \| \| \bar{\mf{e}}(t) \| $,  
with $\lim_{t\to\infty} \|\bm{\Psi}(t) \|= 0$ due to $\lim_{t\to\infty} (\mf{P}(t) - \mf{P}_i(t)) = \mf{0}$, according to Theorem \ref{th:p-inf}. Thus, by \cite[Lemmas 9.4 to 9.6]{Khalil2002}, if the origin is an asymptotically stable equilibrium for the nominal system, the same conclusion applies for the perturbed one.

Similarly, for the aggregate error covariance dynamics \eqref{eq:cov-dyn} we can write 
\begin{equation}
\begin{aligned}
\dot{\bm{\mathcal{P}}}(t) &= \mf{A}^\ast_{\mf{P}}(\kappa, t) \bm{\mathcal{P}}(t) + \bm{\mathcal{P}}(t) \mf{A}^\ast_{\mf{P}}(\kappa,t)^\top + \mf{U}_N \otimes \mf{W}(t)\\ &+ \diag_{i=1}^N(\mf{K}_{\mf{P},i}(t)) \mf{R}(t) \diag_{i=1}^N(\mf{K}_{\mf{P},i}(t))^\top + \mf{G}(\bm{\mathcal{P}}(t)),
\end{aligned}
\end{equation}
where the perturbation $\mf{G}(\bm{\mathcal{P}}(t))$ is due to the mismatch between $\mf{P}(t)$ and $\mf{P}_i(t)$, which appears in $\mf{A}^\ast(\kappa,t)$ and $\mf{K}_i(t)$:
\begin{equation}
\mf{G}(\bm{\mathcal{P}}(t)) = \bm{\Psi}(t) \bm{\mathcal{P}}(t)  + \bm{\mathcal{P}}(t) \bm{\Psi}(t)^\top + \bm{\Omega}(t),
\end{equation}
with the same $\Psi(t)$ from \eqref{eq:psi} and 
\begin{equation}
\begin{aligned}
    \bm{\Omega}(t) &= \diag_{i=1}^N(\mf{K}_{i}(t) \mf{R}_i(t) \mf{K}_{i}(t)^\top -\mf{K}_{\mf{P},i}(t) \mf{R}_i(t) \mf{K}_{\mf{P},i}(t)^\top)  \\
    &= \diag_{i=1}^N(N^2 (\mf{P}_i(t) - \mf{P}(t)) \mf{C}_i(t)^\top \mf{R}_i(t)^{-1} \mf{C}_i(t) (\mf{P}_i(t)  - \mf{P}(t))).
\end{aligned}
\end{equation}

Then, we have
 $   \| \mf{G}(\bm{\mathcal{P}}(t)) \| \leq 2 \| \bm{\Psi}(t) \| \| \bm{\mathcal{P}}(t) \| + \| \bm{\Omega}(t) \|$,
again with $$\lim_{t\to\infty} \| \bm{\Psi}(t) \| = 0, \ \lim_{t\to\infty} \| \bm{\Omega}(t) \|=0,$$ taking into account the convergence $\mf{P}_i(t) \rightarrow \mf{P}(t)$ from Theorem \ref{th:p-inf}. Since the perturbation $\mf{G}(\bm{\mathcal{P}}(t))$ vanishes over time, if the dynamics of the nominal system are stable, by a similar argument of continuity of solutions from \cite[Theorem 9.1]{Khalil2002} we have that both systems converge to the same solution.
\end{proof}

\subsection{Proof of Theorem \ref{th:main}}\label{sec:proof-th}

With the previous results, we can finally prove Theorem \ref{th:main} as follows. In Lemma \ref{lem:stable} and Corollary \ref{cor:kappa}, we have established a choice of consensus gain $\kappa$ to stabilize the estimation error dynamics, under the assumption that $\mf{P}_i(t) = \mf{P}(t)$. By Lemma \ref{lem:comparison}, the same choice of $\kappa$ also serves to stabilize the error dynamics under the actual matrix $\mf{P}_i(t)$ computed online by the ODEFTC algorithm as \eqref{eq:dkb-p}, recalling that $\mf{P}_i(t) \rightarrow \mf{P}(t)$ asymptotically according to Theorem \ref{th:p-inf}. Similarly, recall that Lemma \ref{lem:proof-with-p} proves that the claims in Theorem \ref{th:main} hold under the assumption of $\mf{P}_i(t) = \mf{P}(t)$. Again, from Lemma \ref{lem:comparison}, we can conclude that the results also hold for the original system, completing the proof of Theorem \ref{th:main} and showing that the true error covariance of the estimates also fulfills $\bm{\mathcal{P}}_i(t) \rightarrow \mf{P}(t)$.

\section{Simulation Example}
We provide a short simulation example to illustrate our findings. Consider a LTV system of the form \eqref{eq:sys} with
\begin{equation}
\mf{A}(t) =   \begin{bmatrix}
		0 & 0 & 0.5 \sin (t) & 0\\	
		0 & 0 & 0 & \sin (3t)\\
		0 & 0 & 0.5 \cos(t) & 0\\
		0 & 0 & 0 & \cos(3t)\\
		\end{bmatrix}, \ 
\mf{W} = \begin{bmatrix}
		0 & 0 & 0 & 0\\	
		0 & 0 & 0 & 0\\
		0 & 0 & 1 & 0\\
		0 & 0 & 0 & 1\\
		\end{bmatrix}
\end{equation}
observed by a network of $N=7$ sensors. The communication graph is shown in Figure \ref{fig:graph}. 
\begin{figure}
    \centering
    \includegraphics[width=0.5\columnwidth]{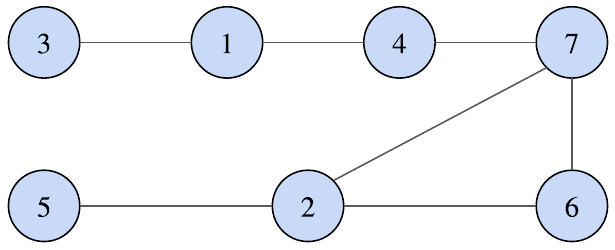}
    \caption{Communication graph $\mathcal{G}$.}
    \label{fig:graph}
\end{figure}
The local measurements of the nodes are given by \eqref{eq:meas} along with $\mf{C}_1 = \mf{C}_5 = \mf{C}_7 = [1, 0, 0, 0]$, $\mf{C}_2 = \mf{C}_4 = [0, 1, 0, 0,]$, $\mf{C}_3 = \mf{C}_6 = [0, 0, 1, 0]$, $\mf{R}_1(t) = \mf{R}_4(t) = \mf{R}_7(t) = 0.05 + 0.01 \sin (0.1 t)$, $\mf{R}_2 = \mf{R}_6 = 0.02$, $\mf{R}_3 = 0.03 + 0.01 \cos (0.1 t)$, $\mf{R}_5 = 0.035$.

We set the parameters of the ODEFTC algorithm as $\alpha = 20$, $\gamma = 0.7$ and $\xi = 10$ (which is an admissible value according to Lemma \ref{lem:Z}). Note that we have $\lambda_\mathcal{G} = 0.2679$, $\| \mf{C}(t) \| \leq 1.73 =: c$ and $\| \mf{R}(t) \| \geq 0.04 =: r_1$. Then, by Corollary \ref{cor:kappa}, a sufficient value of the consensus gain is given by $\kappa > 139.98$.

We have simulated the system using the Euler-Maruyama approach with discretization step $h=10^{-4}$, over 100 realizations of process and measurement noises. We have initialized the centralized filter to $\hat{\mf{x}}(0) = \mf{x}_0=[0,0,0,0]^\top$, $\mf{P}(0) = \mf{P}_0 = \mf{I}_n$, and ODEFTC at random to verify its performance even under unknown initial conditions. For those sequences, we have obtained estimates with the centralized Kalman-Bucy filter, and with ODEFTC using different values of $\kappa$. We computed the mean square error (MSE) at each time instant over the 100 noise realizations. For ODEFTC, we compute and plot the MSE at each node separately. Figure \ref{fig:mse} summarizes the results, showing that even for smaller values of $\kappa$ the estimation error remains stable, and that the results resemble those of the centralized filter as $\kappa$ is increased, validating the results from Theorem \ref{th:main}.
\begin{figure}
    \centering
    \includegraphics[width=0.6\columnwidth]{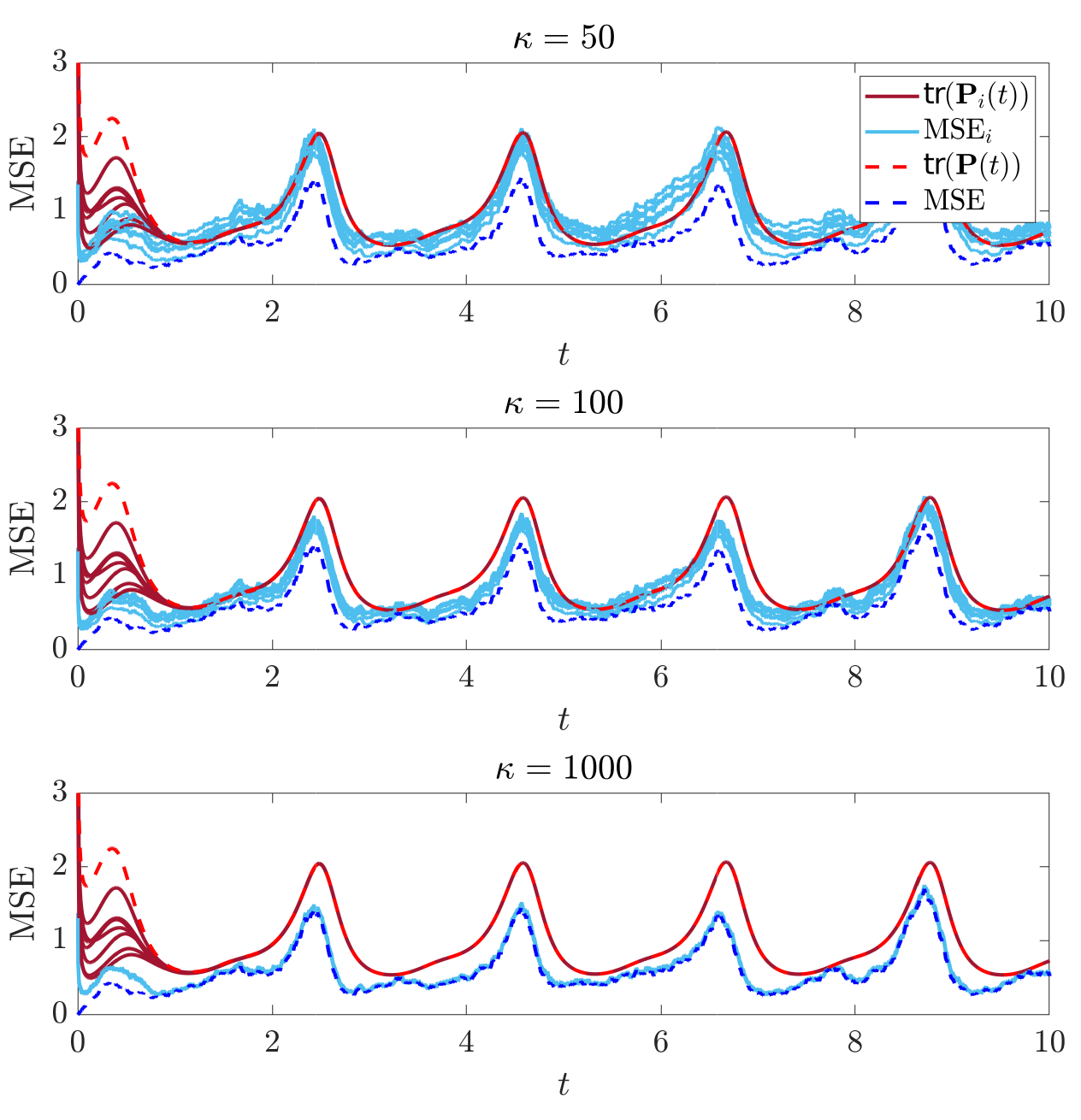}
    \caption{MSE over 100 noise realizations for ODEFTC (all nodes are plotted, continuous line) vs. for the centralized Kalman-Bucy filter (dashed line).}
    \label{fig:mse}
\end{figure}

Moreover, recall that ODEFTC subsumes the proposal of \cite{Battilotti2020} for the LTI case. In that work, a lower bound for a stabilizing $\kappa$ was also obtained. In addition to being applicable to LTV systems, we want to show that our bound for $\kappa$ is also less conservative than that of \cite{Battilotti2020} in the LTI case. Consider the same system matrices as before, but with
\begin{equation}
\mf{A} =   \begin{bmatrix}
		0 & 0 & 0.5 & 0\\	
		0 & 0 & 0 & 1\\
		0 & 0 & 0.5 & 0\\
		0 & 0 & 0 & 1\\
		\end{bmatrix},
\end{equation}
$\mf{R}_1 = \mf{R}_4 = \mf{R}_7 = 0.06$, $\mf{R}_3 = 0.03$, so that the system is now LTI.
The bound in \cite{Battilotti2020} is given by
\begin{equation}\label{eq:kappa-battilotti}
\begin{aligned}
    \kappa &> ({\| \mf{P}_\infty^{-1}\mf{A} + \mf{A}^\top \mf{P}_\infty^{-1} \| + \eta})/{\lambda_\mathcal{G}}, \\
    \eta &= 4N^2 \lambda_{\max}((\mf{P}_\infty^{-1}\mf{W}\mf{P}_\infty^{-1} + \mf{G})  \lambda_{\max}(\mf{G}),
\end{aligned}
\end{equation}
with $\mf{G} = \mf{C}^\top\mf{R}^{-1}\mf{C}$ and $\mf{P}_\infty$ being the steady-state solution of \eqref{eq:ckb} in the LTI case (see \cite[Theorem 5]{Kalman1961}).
Then, the bound from \cite{Battilotti2020} yields $\kappa > 12931.43$, while our bound from Corollary \ref{cor:kappa} results in $\kappa > 93.32$ for this example, providing a smaller lower bound to guarantee stability. 

\section{Conclusions}
In this technical note, we have proven the optimality of the ODEFTC algorithm proposed in \cite{odeftc} to provide distributed state estimates for LTV stochastic systems. It can recover the solution of the centralized Kalman-Bucy filter, which is the optimal filter in this context, as the value of the consensus gain $\kappa$ is increased. Moreover, we provided a sufficient value for $\kappa$ to stabilize the estimator, which can be easily computed a priori by taking into account the measurement model and the topology of the communication graph. This sufficient bound is also less conservative than the one proposed in prior work, allowing to guarantee the stability of the estimator for smaller values of the consensus gain.

\appendix 
\section*{Appendix}
In the subsequent steps, we will make use of the following property for the vectorization operator: $\vect(\mf{A} \mf{B} \mf{C}) = (\mf{C}^\top \otimes \mf{A}) \vect (\mf{B})$, for matrices $\mf{A}, \mf{B}, \mf{C}$ with appropriate dimensions.
Then, we can evaluate the terms of \eqref{eq:lyap-rho2} as follows:
\begin{equation}
\begin{aligned}
    &\bm{\uprho}(t)^\top (\mf{N}(t) \otimes \mf{N}(t)) \bm{\sigma}(t) \\&= \bar{\bm{\uprho}}(t)^\top (\mf{N}(t) \otimes \mf{N}(t)) \bm{\sigma}(t) + \tilde{\bm{\uprho}}(t)^\top (\mf{N}(t) \otimes \mf{N}(t)) \bm{\sigma}(t) \\&= \tilde{\bm{\uprho}}(t)^\top (\mf{N}(t) \otimes \mf{N}(t)) \bm{\sigma}(t)
\end{aligned}
\end{equation}
due to 
\begin{equation}
\begin{aligned}
    &\bar{\bm{\uprho}}(t)^\top (\mf{N}(t) \otimes \mf{N}(t)) \bm{\sigma}(t) \\&= \bm{\uprho}(t)^\top((\mf{U}_N \otimes \mf{I}_n) \otimes \mf{I}_{Nn}) (\mf{N}(t) \otimes \mf{N}(t)) \bm{\sigma}(t) \\
    &= \bm{\uprho}(t)^\top((\mf{U}_N \otimes \mf{I}_n)\mf{N}(t) \otimes \mf{N}(t)) \bm{\sigma}(t) \\
    &= \bm{\uprho}(t)^\top \vect(\mf{N}(t) \bm{\Sigma}(t) \mf{N}(t)(\mf{U}_N \otimes \mf{I}_n)),
\end{aligned}
\end{equation}
and noting that 
\begin{equation}
\begin{aligned}
    &\bm{\Sigma}(t) \mf{N}(t)(\mf{U}_N \otimes \mf{I}_{n}) = \bm{\Sigma}(t)(\mf{I}_N \otimes \mf{P}(t)^{-1})(\mf{U}_N \otimes \mf{I}_n) \\
    &= \bm{\Sigma}(t)(\mf{U}_N \otimes \mf{P}(t)^{-1}) = \bm{\Sigma}(t)(\mf{U}_N \otimes \mf{I}_n)(\mf{I}_N \otimes \mf{P}(t)^{-1}) \\
    &=\bm{\Sigma}(t)(\mathds{1}_N \otimes \mf{I}_n)(\mathds{1}^\top_N \otimes \mf{I}_n)(\mf{I}_N \otimes \mf{P}(t)^{-1}) = \mf{0}
\end{aligned}
\end{equation}
according to $\bm{\Sigma}(t)(\mathds{1}_N \otimes \mf{I}_n) = \mf{0}$ from Lemma \ref{lem:sigma-ort}.

We also have, naming $\mf{M}(t) = \mf{N}(t) \otimes \mf{H} \otimes \mf{I}_n + \mf{H} \otimes \mf{I}_n \otimes \mf{N}(t)$ for brevity,
\begin{equation}\label{eq:terms-rho}
\begin{aligned}
\bm{\uprho}(t)^\top\mf{M}(t)\bm{\uprho}(t) &= \bar{\bm{\uprho}}(t)^\top\mf{M}(t)\bar{\bm{\uprho}}(t) + \tilde{\bm{\uprho}}(t)^\top\mf{M}(t)\tilde{\bm{\uprho}}(t)  + 2 \tilde{\bm{\uprho}}(t)^\top\mf{M}(t)\bar{\bm{\uprho}}(t).
\end{aligned}
\end{equation}
We can show that all terms containing $\bar{\bm{\uprho}}(t)$ are zero. In particular, we have
\begin{equation}\label{eq:term-bar-rho}
\begin{aligned}
    &\mf{M}(t)\bar{\bm{\uprho}}(t) = (\mf{N}(t) \otimes \mf{H} \otimes \mf{I}_n) \bar{\bm{\uprho}}(t)  +  (\mf{H} \otimes \mf{I}_n \otimes \mf{N}(t)) \bar{\bm{\uprho}}(t),
\end{aligned}
\end{equation}
where
\begin{equation}\label{eq:barrho1}
\begin{aligned}
&(\mf{H} \otimes \mf{I}_n \otimes \mf{N}(t)) \bar{\bm{\uprho}}(t) \\ 
&= ((\mf{H} \otimes \mf{I}_n) \otimes \mf{N}(t)) ((\mf{U}_N \otimes \mf{I}_n) \otimes \mf{I}_{Nn}) {\bm{\uprho}}(t) \\
&= ((\mf{H} \otimes \mf{I}_n)(\mf{U}_N \otimes \mf{I}_n) \otimes \mf{N}(t)) {\bm{\uprho}}(t) \\
&= ((\mf{HU}_N \otimes \mf{I}_n) \otimes \mf{N}(t)) {\bm{\uprho}}(t) = \mf{0},
\end{aligned}
\end{equation}
recalling that $\mf{HU}_N = \mf{0}$. 
Defining $\bar{\mf{X}}(t)$ such that $\bar{\bm{\uprho}}(t) = \vect(\bar{\mf{X}}(t))$, note that 
\begin{equation}\label{eq:barrho2}
\begin{aligned}
&(\mf{H} \otimes \mf{I}_n \otimes \mf{N}(t)) \bar{\bm{\uprho}}(t) = \vect (\mf{N}(t) \bar{\mf{X}}(t) (\mf{H} \otimes \mf{I}_n)) = \mf{0} \\ &\implies \mf{N}(t) \bar{\mf{X}}(t) (\mf{H} \otimes \mf{I}_n) = \mf{0}  \\ &\implies (\mf{N}(t) \bar{\mf{X}}(t) (\mf{H} \otimes \mf{I}_n))^\top = (\mf{H} \otimes \mf{I}_n) \bar{\mf{X}}(t) \mf{N}(t) = \mf{0} \\& \implies \vect ((\mf{H} \otimes \mf{I}_n) \bar{\mf{X}}(t) \mf{N}(t)) = (\mf{N}(t) \otimes \mf{H} \otimes \mf{I}_n)\bar{\bm{\uprho}}(t) =  \mf{0},
\end{aligned}
\end{equation}
where we have used the fact that the matrices involved are symmetric. Taking into account \eqref{eq:barrho1}-\eqref{eq:barrho2}, we have that $\mf{M}(t)\bar{\bm{\uprho}}(t) = \mf{0}$, and therefore the terms $\bar{\bm{\uprho}}(t)^\top\mf{M}(t)\bar{\bm{\uprho}}(t)$ and $\tilde{\bm{\uprho}}(t)^\top\mf{M}(t)\bar{\bm{\uprho}}(t)$ in \eqref{eq:terms-rho} are also zero.

Finally, only the disagreement term remains, with
\begin{equation}
\begin{aligned}
\tilde{\bm{\uprho}}(t)^\top\mf{M}(t)\tilde{\bm{\uprho}}(t) &= \tilde{\bm{\uprho}}(t)^\top (\mf{N}(t) \otimes \mf{H} \otimes \mf{I}_n) \tilde{\bm{\uprho}}(t)   + \tilde{\bm{\uprho}}(t)^\top (\mf{H} \otimes \mf{I}_n \otimes \mf{N}(t)) \tilde{\bm{\uprho}}(t).
\end{aligned}
\end{equation}
Note that
\begin{equation}
\begin{aligned}
&\tilde{\bm{\uprho}}(t)^\top (\mf{H} \otimes \mf{I}_n \otimes \mf{N}(t)) \\
&= \bm{\uprho}(t)^\top ((\mf{H} \otimes \mf{I}_n) \otimes \mf{I}_{Nn}) ((\mf{H} \otimes \mf{I}_n) \otimes \mf{N}(t)) \\
&= \bm{\uprho}(t)^\top ((\mf{H} \otimes \mf{I}_n)(\mf{H} \otimes \mf{I}_n) \otimes \mf{N}(t)) \\
&= \bm{\uprho}(t)^\top ((\mf{H} \otimes \mf{I}_n) \otimes \mf{N}(t)) \\
&= \bm{\uprho}(t)^\top ((\mf{H} \otimes \mf{I}_n) \otimes \mf{I}_{Nn}) (\mf{I}_{Nn}\otimes \mf{N}(t)) \\
&= \tilde{\bm{\uprho}}(t)^\top(\mf{I}_{Nn}\otimes \mf{N}(t)),
\end{aligned}
\end{equation}
where we have used the fact that $\mf{H}^2 = \mf{H}$. Moreover, we have the following equivalences, letting $\tilde{\bm{\uprho}}(t) = \vect (\tilde{\mf{X}}(t))$:
\begin{equation}
\begin{aligned}
&(\tilde{\bm{\uprho}}(t)^\top (\mf{H} \otimes \mf{I}_n \otimes \mf{N}(t)))^\top = (\tilde{\bm{\uprho}}(t)^\top(\mf{I}_{Nn}\otimes \mf{N}(t)))^\top \\
& \implies (\mf{H} \otimes \mf{I}_n \otimes \mf{N}(t)) \tilde{\bm{\uprho}}(t) = (\mf{I}_{Nn}\otimes \mf{N}(t)) \tilde{\bm{\uprho}}(t) \\
&\implies \vect(\mf{N}(t) \tilde{\mf{X}}(t) (\mf{H}\otimes \mf{I}_n)) = \vect(\mf{N}(t) \tilde{\mf{X}}(t)) \\
&\implies (\mf{N}(t) \tilde{\mf{X}}(t) (\mf{H}\otimes \mf{I}_n))^\top = (\mf{N}(t) \tilde{\mf{X}}(t))^\top \\
&\implies ((\mf{H}\otimes \mf{I}_n) \tilde{\mf{X}}(t) \mf{N}(t)) = \tilde{\mf{X}}(t) \mf{N}(t) \\
&\implies  \vect((\mf{H}\otimes \mf{I}_n) \tilde{\mf{X}}(t) \mf{N}(t)) = \vect(\tilde{\mf{X}}(t) \mf{N}(t)) \\
&\implies (\mf{N}(t) \otimes (\mf{H} \otimes \mf{I}_n))\tilde{\bm{\uprho}}(t) = (\mf{N}(t) \otimes \mf{I}_{Nn}) \tilde{\bm{\uprho}}(t). 
\end{aligned}
\end{equation}
With these results, and recalling that the other terms become zero, we have
\begin{equation}
\begin{aligned}
\bm{\uprho}(t)^\top\mf{M}(t)\bm{\uprho}(t) &= \tilde{\bm{\uprho}}(t)^\top (\mf{N}(t) \otimes \mf{I}_{Nn} + \mf{I}_{Nn} \otimes \mf{N}(t)) \tilde{\bm{\uprho}}(t).
\end{aligned}
\end{equation}


\bibliographystyle{IEEEtran}
\bibliography{biblio.bib}

\end{document}